\documentclass[a4paper,11pt,reqno]{amsart}
\usepackage{amsfonts}
\usepackage{amsmath}
\usepackage{logicproof}
\usepackage{amssymb, amsthm}
\usepackage{tabularx}
\usepackage{caption}
\usepackage{booktabs}
\usepackage{multirow,dsfont}
\usepackage{array}
\usepackage{floatrow}
\usepackage{float}
\usepackage{floatpag}
\usepackage{centernot}
\usepackage{enumitem}
\newcolumntype{L}[1]{>{\raggedright\let\newline\\\arraybackslash\hspace{0pt}}m{#1}}
\newcolumntype{C}[1]{>{\centering\let\newline\\\arraybackslash\hspace{0pt}}m{#1}}
\newcolumntype{R}[1]{>{\raggedleft\let\newline\\\arraybackslash\hspace{0pt}}m{#1}}
\usepackage[symbol]{footmisc}
\usepackage[labelfont=bf,format=plain,justification=raggedright,singlelinecheck=false]{caption}
\usepackage{mathrsfs}
\usepackage[]{mdframed}
\usepackage[colorlinks]{hyperref}
\usepackage{tcolorbox}
\usepackage{adjustbox}
\tcbuselibrary{theorems}
\newtcbtheorem[number within=section]{mytheo}{Note}
{colback=white!5,colframe=black!35!black,fonttitle=\bfseries}{th}

\DeclareMathOperator*{\esssup}{ess\,sup}
\DeclareMathOperator*{\essinf}{ess\,inf}

\numberwithin{equation}{section}

\usepackage{graphicx}
\usepackage{textgreek}

 {
      \theoremstyle{plain}
      
  }
\newtheorem{theorem}{Theorem}[section]

\newtheorem{prop}[theorem]{Proposition}
\newtheorem{lemma}[theorem]{Lemma}

\def\CC{\mathbb C}
\def\NN{\mathbb N}
\def\ZZ{\mathbb Z}

\def\RR{\mathbb R}

\begin{document}
\title[Gabor frames generated by Random-Periodic time-frequency shifts]{Gabor frames generated by Random-Periodic time-frequency shifts}
\author[R. Sarthak]{Sarthak Raj}
\address[R. Sarthak]{Department of Mathematics, Indian Institute of Technology Delhi, India}
\email{sarthakraj.math@gmail.com}
\author{S. Sivananthan}
\address[S. Sivananthan]{Department of Mathematics, Indian Institute of Technology Delhi, India}
\email{siva@maths.iitd.ac.in}
\begin{abstract}
    In this article, we consider a variation of the existence of Gabor frames in a probabilistic setting, in which we consider time-frequency shifts taken over random-periodic sets. We demonstrate that the method of selecting random-periodic time-frequency shifts is successful with high probability for specific categories of well-behaved functions, notably including Hermite functions, totally positive functions, and B-spline functions. In particular, we show that if $x_1, x_2, \ldots ,x_m$ are independent and uniformly  distributed in $[0,1),$ with $m$ sufficiently large, then the set of time-frequency shifts $\Lambda \times \ZZ, $ where $\Lambda=\ZZ + \{x_1, x_2, \ldots, x_m\},$ forms Gabor frame with high probability.
\end{abstract}
\maketitle

\section{Introduction}
Ever since its inception, Gabor frames have been not only of just theoretical importance but also a tool of vital importance in different fields, with myriad applications. The existence of Gabor frames with ``nice" generators under certain ``time-frequency" shifts has been a problem that has been considered under different variations. We've considered one such variation of the existence of the Gabor frame problem, in which we establish the frame property under a choice of time-frequency shifts that are periodically random, yet structurally simple. We demonstrate that, with high probability, for generators meeting specific decay conditions, selecting a sufficient number of points from the unit interval and applying integer shifts at these points, when coupled with integers, results in time-frequency shifts that constitute a Gabor frame.

\subsection{Related Literature} Given a Hilbert space $H,$ one question that is asked more often then not is given an element in $H$ can we get a representation of it out of a sequence of elements in $H?$ The motivation of representing something usually originates out of our natural instincts, which compels us to see something as a combination of known things. One answer to this is given by, what we call, an orthonormal basis. Using orthonormal basis, one gets an immediate representation (unique) of any element in $H.$ On one hand, an orthonormal basis gives us an easy way to represent an element; there is a flip side to it. When it comes to real-world applications, the structure of orthonormal basis is very stringent, and often it is required to have some flexibility when applications are considered. The notion of frames overcomes this difficulty. 

A sequence $\{h_{n}\}_{n \in \NN}$ of elements in $H$ is said to form a frame for $H$, if there are constants $ 0<A<B $ such that
\begin{equation*}
    A\|h\|^{2} \leq \sum_{n \in \NN} |\langle h,h_{n} \rangle|^{2} \leq B\|h\|^{2}, \hspace{0.5cm}\forall h \in H.
\end{equation*}
It is not apparent from the definition, but frames provide a robust basis-like representation of elements in $H.$ These representations are generally non-unique, and this is what makes them more suitable for application purposes. The concept of frames was introduced by Duffin and Schaeffer \cite{dnsframe} in 1952 for studying non-harmonic Fourier series. It didn't pick up as much attention as it did after it was reintroduced by Young in his book \cite{youngboook}. Since then, it has found wide applications in a variety of fields. Notably, they are used in sampling theory \cite{AkGroNonUniSampSiam, Ortega}, wavelet theory \cite{PNEID,TLWID}, sparse approximation \cite{nonlinappnielsen}, pseudodifferential operators \cite{GroHeilMSPDO}, data transmission with loss \cite{goyalframewither}, among many other applications. 

Gabor frames are a particular type of frame. The Gabor systems are generated by two fundamental operators: 
\begin{equation*}
    \mbox{\textbf{Translation:} for any } a \in \RR, T_{a}: L^{2}(\RR) \rightarrow L^{2}(\RR), T_{a}f(x)=f(x-a),
\end{equation*}
\begin{equation*}
    \mbox{\textbf{Modulation:} for any } b \in \RR, M_{b}:L^{2}(\RR) \rightarrow L^{2}(\RR), M_{b}f(x)=e^{2 \pi i b x} f(x).
\end{equation*}
Let $ \Lambda $ be a sequence of points in $\RR^{2},$ then the \textit{time-frequency} shifts of $g \in L^{2}(\RR)$ with respect to $\Lambda $ is defined to be $M_{v}T_{u}g$ for $(u,v) \in \Lambda.$ Then the Gabor system with respect to $\Lambda$ is defined to be 
$$\mathcal{G}(g, \Lambda ):=\{M_{v} T_{u}g: (u,v) \in \Lambda \}=\{e^{2 \pi i v \cdot}g(\cdot - u): (u,v) \in \Lambda\}.$$
When the set $\mathcal{G}(g,\Lambda)$ forms a frame for $L^{2}(\RR),$ then it is called a Gabor frame.

Historically, one choice of $\Lambda$ has been of particular interest to researchers in Gabor analysis. This choice is rectangular lattices, which is  $\Lambda = \alpha \ZZ \times \beta \ZZ,$ with $\alpha, \beta >0,$ and collection of all such $\alpha, \beta > 0$ such that $\mathcal{G}(g, \alpha \ZZ \times \beta \ZZ )$ forms a frame is called the frame set for $g.$ 

It was first claimed by von Neumann \cite{voneumanncompleteness} that the set of time-frequency shifts of the Gaussian window function $\varphi(t)=2^{1/4} e^{-\pi t^2}$ along $\Lambda = \ZZ^2$ is complete in $L^2(\RR),$ i.e., the linear span of $\mathcal{G}(\varphi,\Lambda)$ is dense. Dennis Gabor furthered this claim and conjectured that every function $f \in L^{2}(\RR)$ can be represented as following: 
$$f =\sum_{k,n \in \ZZ} c_{k,n} M_{k}T_{n} \varphi.$$ 
Janssen proved that such a series expansion exists, but the convergence holds in the sense of tempered distributions, not in $L^{2}-$norm \cite{janssengaborrep}.

The theory of Gabor analysis is extensive but the central question of Gabor analysis is how to determine the frame set for a function in $L^{2}(\RR).$ This innocuous-looking question is highly non-trivial. The necessary condition $\alpha \beta \leq 1 $ is well known, but the sufficient condition had been proven only for a handful of cases. It was proved for Gaussian by Seip and Wallst\'en \cite{SeipDensitythm1, SeipDensitythm2} and Lyubarski\u{\i} \cite{Lyubarskigaussian}. It has also been proved for hyperbolic secant \cite{janssenhyperbolicsecant}, two sided exponential \cite{janssentwosidedexp}, totally positive functions of finite type \cite{Grochenigtpfinite}, totally positive functions on rational lattices \cite{Grochenigtpoverrationallattice}. Based on these observations, Gr\"ochenig \cite{grochenigmysgabframes} conjectured the following : 
$$\mbox{For even Hermite functions, the frame set is } \{(\alpha, \beta) : \alpha \beta < 1\} \mbox{, and }   $$
$$\mbox{for odd Hermite functions, the frame set is } \{(\alpha, \beta) : \alpha \beta < 1, \alpha \beta \neq 1-\frac{1}{N}, N=2,3,\ldots \}.$$
The conjecture was proved to be false by Lemvig \cite{Lemvigcounterhermite} for the cases of Hermite functions with $n=4m+2$ and $4m+3.$ He further gave numerical evidence for the case $n=4$ and $5$ to be false. This led him to believe that the conjecture would fail for the cases $n=4m$ and $4m+1.$ It was not until very recently that the conjecture was proved to be false for the remaining cases too, except for the case $n=1$, by Lemvig et al. \cite{Lemvigetalcounterhermite}.

A common thread that runs through all these results is that they are all existence results, i.e., they establish the frame property, but they fail to provide explicit frame constants.

In this article, we take a probabilistic approach to the problem. Rather than relying on a rectangular lattice, we introduce semi-regular sampling patterns, which we refer to as random-periodic patterns. This approach is motivated by the negative results concerning the class of Hermite functions and our belief that selecting a sufficient number of points in the unit interval—then shifting them along integers on an axis and coupling it with integers—will produce a sufficiently dense pattern to exhibit frame behavior, at least for well-behaved functions. We were able to prove that this happens with high probability under certain assumptions on the window function. Also, at the expense of deterministic setting, we provide explicit frame constants. 

Recently, Romero et al. \cite{romerorandomperiodic} established a sampling result for shift-invariant spaces using such random-periodic sampling patterns. They considered  signals spanned by the integer shifts of a set of generating functions with distinct frequency profiles and demonstrated that sampling on a random-periodic set enables high probability recovery, provided that the sampling density exceeds the number of frequency profiles by a logarithmic factor. This framework, in particular, includes bandlimited functions with multi-band spectra.
\subsection{Organisation}
This article is organized as follows. Section 2 provides the necessary background for the subsequent sections. Specifically, we begin with a brief overview of the Zak transform, which plays a crucial role in the main result, followed by an introduction to Hoeffding's inequality. Section 3 is dedicated to proving the main result. Finally, in Section 4, we discuss relevant examples.  
\section{Preliminaries}

\subsection{Fourier Transform} The Fourier transform of $f \in L^{1}(\RR)$ is defined by the function 
\begin{equation*}
    \mathcal{F}(f)(\xi)=\widehat{f}(\xi):= \int_{\RR} f(t)e^{- 2 \pi i \xi t} dt, ~ \xi \in \RR.
\end{equation*}
 
This definition is then extended, in its usual sense, to define the Fourier transform for the functions in $L^{2}(\RR).$

\subsection{The Zak Transform}
In this subsection, we study a classical transform in time-frequency analysis called the Zak transfrom.   The Zak transform of a function $f \in L^{2}(\RR)$ is formally defined as a function of two variables through the series:
\begin{equation}
\label{defzak}
    (Zf)(t,\xi)=  \sum_{k \in \ZZ}f(t+k)e^{2 \pi i k \xi}, \hspace{0.5cm} t,\nu \in \RR.
\end{equation}
This transform is strikingly informative in Gabor analysis. In fact, in certain cases, it is possible to give complete characterizations for the Gabor systems forming bases (orthonormal, Riesz) using just the Zak transform \cite{oleframes}.\par
Although, in the present context, it is customary to refer to it as the Zak transform, its inception can be traced back to the work of Gelfand \cite{gelfandeigenex} on eigenfunction expansions associated with Schrödinger operators with periodic potentials. Subsequently, it has been rediscovered in many different contexts. Weil \cite{weiltransform} defined it on arbitrary locally compact abelian groups with respect to arbitrary closed subgroups. J. Zak \cite{zakkqrep} used it in solid-state physics, and there he referred to it as k-q representation. \par
While interpreting equation (\ref{defzak}), we have to be careful. For instance, the series on the right-hand side of (\ref{defzak}) converges pointwise if the function $f$ is taken to be a continuous function with compact support, but for a general function, $f \in L^{2}(\RR),$ more precise arguments have to be made. To this end, we mention the following result, which asserts that the series defining $Zf$ converges in $L^{2}(Q)$, where $Q:=[0,1)\times [0,1).$
\begin{theorem}
     $Z$ defines a unitary map from $L^{2}(\RR)$ onto $L^{2}(Q).$
\end{theorem}
Sometimes it is required to move from the Zak transform of a function in the time domain to the Zak transform of its Fourier transform. In general, one cannot expect such a transition, but under sufficient decay conditions, it is possible. One such result is the following: 

\begin{prop}
\label{zakswitch}
    Let $W(\RR):=\{f: \RR \rightarrow \CC : \esssup_{x \in [0,1]} \sum_{k \in \ZZ}|f(x-k)| < \infty \}$ denote the Wiener-Amalgam class of functions. If $ f, \widehat{f} \in W(\RR)$ (hence $f, \widehat{f}$ are continuous), then 
    \begin{equation*}
        Zf(t,\xi)=e^{2 \pi i t \xi}Z\widehat{f}(\xi,-t), \hspace{0.5cm}\mbox{ for all } (t, \xi) \in \RR^{2}.
    \end{equation*}
\end{prop}
For an introduction on the Zak transform, one can refer to the excellent books \cite{TLWID,grobook}.
\subsection{Probability bound for the sum of random variables}
This subsection introduces a technical lemma, which gives a probability bound for the sum of independent random variables that exceeds its expected value by a positive amount. This is famously known as Hoeffding's inequality. 
\begin{lemma}{(Hoeffding's Inequality)} \cite[Theorem 2]{Hoeffdingineq}
\label{hoeffding}
    Let $X_{1}, X_{2}, \ldots, X_{n}$ be independent random variables and $a_{i} \leq X_{i} \leq b_{i}$ for all $i=1,2, \ldots, n.$ If $\overline{X}= \frac{1}{n}\sum_{i=1}^{n} X_{i},$ $\mu = \mathbb{E}(\overline{X}),$ and $t > 0$, then
    \begin{equation}\mathbb{P}(\overline{X}-\mu > t) \leq \exp{\left(\frac{-2 n^2 t^2}{\sum_{i=1}^{n}(b_{i}-a_{i})^{2}}\right)}.\end{equation}
\end{lemma}
\section{Main Results}
To begin this section, we make the following observation: Let $\Lambda $ be translates of integers, i.e., $\Lambda= \ZZ + \{x_{1}, x_{2}, \ldots, x_{m}\}$, where $x_{i} \in [0,1), ~ i=1,2,\ldots, m,$ and let $\mathcal{G}(g, \Lambda \times \ZZ):=\{ M_{k} T_{n+x_i}g : k,n \in \ZZ, i=1,2,\ldots m\}$ be the set containing the time-frequency shifts of a function $g \in L^{2}(\RR)$ with respect to the set $\Lambda \times \ZZ.$
Now, since
\begin{align}
    \begin{split}
        Z(M_{k}T_{n}g)(t,\xi)&=\sum_{l \in \ZZ} M_{k}T_{n}g(t+l) e^{2 \pi i l \xi}\\
        &= \sum_{l \in \ZZ} g(t+l-n) e^{2 \pi i k (t+l)} e^{2 \pi i l \xi}\\
        &= e^{2 \pi i k t} e^{2 \pi i n \xi} Zg(t,\xi),
    \end{split}
\end{align}
which in turn gives
\begin{align}
\begin{split}Z \mathcal{G}(g, \Lambda \times \ZZ ) &= \{ZM_{k} T_{n+x_i}g : k,n \in \ZZ, i=1,2,\ldots m\}\\
&=\{E_{k,n} ZT_{x_i}g : k,n \in \ZZ, i=1,2,\ldots m\},
\end{split}\end{align}
where $E_{k,n}(t,\xi)=e^{2 \pi i k t} e^{2 \pi i n \xi}.$
Using this, we translate the problem of $\mathcal{G}(G, \Lambda \times \ZZ)$ being a frame for $L^2(\RR)$ to a  problem in $L^2(Q).$
\begin{lemma}
\label{equiprob}
    Let $\Lambda= \ZZ +\{x_{1}, x_{2}, \ldots, x_{m}\},$ where $x_{i} \in [0,1)$ for $i=1,2,\ldots, m,$ and $g$ be a function in $L^{2}(\RR).$ If there are constants $A,B > 0$ such that 
    \begin{equation}
    \label{ch2:sufcond}
        A \leq \sum_{i=1}^{m} |Z(T_{x_i}g)(t,\xi)|^{2} \leq B \mbox{ for a.e. } (t,\xi) \in Q,
    \end{equation}   then the Gabor system $\mathcal{G}(g, \Lambda \times \ZZ)$ is a frame for $L^{2}(\RR).$
\end{lemma}
\begin{proof}
For $f \in L^2(\RR),$ consider the sum
\begin{align*}
    \sum_{i=1}^{m} \sum_{k,n \in \ZZ} |\langle f,M_{k}T_{n}T_{x_i}g \rangle_{L^{2}(\RR)}|^{2}&= \sum_{i=1}^{m} \sum_{k,n \in \ZZ} |\langle Zf,Z(M_{k}T_{n}T_{x_i}g)\rangle_{L^{2}(Q)} |^{2} \\ 
&= \sum_{i=1}^{m} \sum_{k,n \in \ZZ} |\langle Zf,E_{k,n} Z(T_{x_i}g)\rangle_{L^{2}(Q)}|^{2} \\
&= \sum_{i=1}^{m} \sum_{k,n \in \ZZ} |\langle Zf \times \overline{Z(T_{x_i}g)},E_{k,n} )\rangle_{L^{2}(Q)}|^{2}
\end{align*}
where the first equality follows from the unitary property of the Zak transform.
    Now using the fact that $\{E_{k,n}: k, n \in \ZZ\}$ forms an orthonormal basis for $L^2(Q),$ we have
    \begin{align*}
    \sum_{i=1}^{m} \sum_{k,n \in \ZZ} |\langle f,M_{k}T_{n}T_{x_i}g\rangle_{L^{2}(\RR)}|^{2} 
   &= \sum_{i=1}^{m} \|Zf \times Z(\overline{T_{x_i}g}) \|^{2}\\
   &= \sum_{i=1}^{m} \int_{Q} |Zf(t,\xi)|^{2} |Z(T_{x_i}g)(t,\xi)|^{2} dt d \xi \\
   &= \int_{Q} |Zf(t,\xi)|^{2}(\sum_{i=1}^{m}|Z(T_{x_i}g)(t,\xi)|^{2})dt d \xi.
\end{align*}
Since, we have that there are $A,B>0$ such that 
$$ A \leq \sum_{i=1}^{m} |Z(T_{x_i}g)(t,\xi)|^{2} \leq B \mbox{ for a.e. } (t,\xi) \in Q, $$
then 
$$ A \|f\|^2 \leq \sum_{i=1}^{m} \sum_{k,n \in \ZZ} |\langle f,M_{k}T_{n}T_{x_i}g\rangle_{L^{2}(\RR)}|^{2} \leq B \|f\|^{2}.$$
\end{proof} 
To achieve such a deterministic result, we must select points that satisfy (\ref{ch2:sufcond}). However, determining these points manually would be an arduous task, especially since the required number of points is not known in advance. To circumvent this difficulty, we approach this inequality probabilistically by considering the points $x_{i}$ as random variables.

Consider a set of points $x_{1}, x_{2}, \ldots, x_{m}$ that are uniformly and independently distributed in the unit interval $ [0,1).$ For each $i=1,2,\ldots,m ,$ the functions $|Z(T_{x_{i}}g)(x,\xi)|^{2}$ determine a random variable for all $(x,\xi) \in Q.$ Consequently, $(\ref{ch2:sufcond})$ constitutes an event for us. We call this event E, 
 \begin{equation}
 \label{eventE}\mbox{E } \equiv A \leq \sum_{i=1}^{m} |Z(T_{x_i}g)(t,\xi)|^{2} \leq B \mbox{ for a.e. } (t,\xi) \in Q .\end{equation}
Simply put, our next result—which is also the main result of this article—states that the event E occurs with high probability, provided that a sufficient number of points are selected from the unit interval $[0,1)$. The precise formulation follows.
\begin{theorem}
\label{maintheorem}
Let $\{x_1, x_2, \ldots , x_m\}$ be independent random variables distributed uniformly in the unit interval $[0,1).$ Assume that $g \in L^{2}(\RR)$ satsify the following-
\begin{enumerate}
    \item (Periodic Summability)
    \begin{align*}
       K=\sup_{t \in [0,1]} \sum_{l \in \ZZ}|g(t+l)|< \infty\mbox{ and } & K'=\sup_{\xi \in [0,1]} \sum_{l \in \ZZ}|\widehat{g}(\xi+l)|< \infty.
    \end{align*}
    \item (Square summability away from zero).
    There exists $q, R > 0$ such that  
    \begin{equation*}
        q \leq \sum_{l \in \ZZ} |\widehat{g}(\xi + l)|^2 \leq R \mbox{ for all } \xi \in [0,1].
    \end{equation*}
    \item (Lipschitz-type boundedness of Zak transform)
    \begin{equation*}
       C= \sup_{(t,\xi)\neq (t',\xi') \in Q}\frac{|Zg(t,\xi)-Zg(t',\xi')|}{\|(t,\xi)-(t',\xi')\|_{\infty}} < \infty.
    \end{equation*}
\end{enumerate}
If $\alpha < \frac{q}{2}$, $\beta > \frac{q}{2},$ $0 < \varepsilon < 1,$ and \\     $m > \max\{\frac{K^4}{2(\beta-\frac{q}{2})^{2}} \ln{\left(\frac{2(\frac{4CK}{q}+1)^{2}}{\varepsilon}\right)}, \frac{K^4}{2(\alpha-\frac{q}{2})^{2}} \ln{\left(\frac{2(\frac{4CK}{q}+1)^{2}}{\varepsilon}\right)},$ then 
\begin{equation*}
    m \alpha \|f\|^{2} \leq \sum_{i=1}^{m} \sum_{k,n \in \ZZ} |\langle f,M_{k}T_{n+{x_{i}}}g \rangle|^{2} \leq m \beta \|f\|^{2}
\end{equation*}
holds for all $f \in L^{2}(\RR)$ with probability atleast $1-\varepsilon.$
\end{theorem}
A few remarks are in order. First, it can be noted that under the first assumption, the Zak transform $Z{g}(t,\xi)$ is well defined for every $t$ and $\xi$. Also, the hypothesis of Proposition \ref{zakswitch} is satisfied; thus it is allowed to write $Z{g}(t,\xi)= e^{2 \pi i t \xi} Z{\widehat{g}}(\xi,-t)$ for all $t $ and $\xi.$ Also note that the constants $\alpha$ and $\beta$ can be taken sufficiently close to each other, i.e., there ratio is close to one. In engineering literature, this is referred to as \textit{snugness} of the frame. It allows for a numerically efficient recovery of a function.

Let us briefly outline the approach to proving the theorem. The first step is two bring the original problem down to a problem in $L^2(Q).$ This is done in the Lemma $\ref{equiprob}.$ Next, we establish that $(\ref{eventE}),$ with $A= m\alpha$ and $B=m\beta,$ occurs with high probability. The random variables involved in the event E are shown to be bounded, with their expectations also demonstrated to be bounded. Additionally, a covering bound is introduced for a technical resaon which is discussed in the proof. With these preparations in place, we apply Hoeffding’s inequality, leading to our result.
\begin{proof}
     As mentioned in the discussion above this theorem, $|Z(T_{x_{i}}g)(t,\xi)|^{2}$ is a function of a random variable and is therefore a random variable for all $i$ and $(t,\xi).$ So we first calculate the expectation for these random variables. For $i=1,2,\ldots , m$ and $(t, \xi) \in Q$
    \begin{align*}
        \mathbb{E}[|Z(T_{x_{i}}g)(t,\xi)|^{2}] 
        &= \int_{0}^{1} |Z(T_{x}g)(t,\xi)|^{2} dx \\
        &= \int_{0}^{1} |Zg(t-x,\xi)|^{2} dx.
    \end{align*}
    Using the fact that $Zg(t, \xi)= e^{2 \pi i t \xi} Z \widehat{g}(\xi,-t),$ we write the above expression as
    \begin{align*}
        \mathbb{E}[|Z(T_{x_{i}}g)(t,\xi)|^{2}] &= \int_{0}^{1} |Z \widehat{g}(\xi, x-t)|^{2} dx\\
        &= \int_{0}^{1} \left(\sum_{k \in \ZZ} \widehat{g}(\xi +k) e^{2 \pi i k (x-t)}\right) \overline{\left(\sum_{n \in \ZZ} \widehat{g}(\xi +n) e^{2 \pi i n (x-t)}\right)}dx\\
        &= \int_{0}^{1} \sum_{k,n \in \ZZ}  \widehat{g}(\xi + k) \overline{\widehat{g}(\xi +n)}e^{2 \pi i (k-n)(x-t)}dx\\
        &= \sum_{l \in \ZZ}|\widehat{g}(\xi + l)|^{2} \leq R \mbox{ for all $\xi \in [0,1].$ }   
    \end{align*}
    By the assumptions on $g,$ $$|Z(T_{x_{i}}g)(t,\xi)|^{2} \leq K^{2} \mbox{ for all }(t, \xi) \in Q , \mbox{ and }  i=1,2,\ldots ,m.$$
    The above arguments establish the boundedness of the random variables and their expectations.
    
    Note that event E is equal to E$_{1} \cap$E$_{2},$ where 
   \begin{equation*}\mbox{E}_{1} \equiv \esssup_{(t,\xi)\in Q}\sum_{i=1}^{m} |Z(T_{x_i}g)(t,\xi)|^{2} \leq B \mbox{ and } \mbox{E}_{2}\equiv \essinf_{(t,\xi)\in Q}\sum_{i=1}^{m} |Z(T_{x_i}g)(t,\xi)|^{2} \geq A.\end{equation*}
   As it is apparent, these events are varying over uncountably many random variables, and to estimate probability in its present form is virtually impossible. To overcome this difficulty, we make use of Assumption $(3)$ along with cover bound to bring down the estimate for uncountably many random variables to finitely many of them.
   
\textbf{Cover Bound}: Consider a finite cover $M_{\delta} $ in $[0,1] \times [0,1]$ such that for any $(t,\xi) \in Q$, there exists $(t', \xi ') \in M_{\delta}$ which satisfies $\|(t, \xi)-(t', \xi ')\|_{\infty} \leq \delta.$ In fact, take $M_{\delta}=  \delta \ZZ^{2} \cap [0,1]\times [0,1]. $ We have chosen such a uniform mesh because it is relatively easier to keep track of the number of points for such a choice of $M_{\delta}$.\par
For any $(t,\xi) \in Q,$ pick $(t', \xi') \in M_{\delta}$ such that $\|(t, \xi)-(t', \xi ')\|_{\infty} \leq \delta.$ Consider the difference
\begin{align*}
    ||Z(T_{x_{i}}g)(t,\xi)|^{2}-|Z(T_{x_{i}}g)(t',\xi ')|^{2}| &\leq (|Z(T_{x_{i}}g)(t,\xi)|+ |Z(T_{x_{i}}g)(t',\xi ')|) \times \\ &|Z(T_{x_{i}}g)(t,\xi)-Z(T_{x_{i}}g)(t',\xi')| \\
& \leq 2K|Z(T_{x_{i}}g)(t,\xi)-Z(T_{x_{i}}g)(t',\xi')|
\end{align*}
By the assumption of Lipschitz type boundedness of the Zak transform, we have
\begin{align*}
||Z(T_{x_{i}}g)(t,\xi)|^{2}-|Z(T_{x_{i}}g)(t',\xi ')|^{2}| &\leq 2 K C \|(t,\xi)-(t', \xi')\|_{\infty}\\
& \leq 2 K C \delta.
\end{align*}
Choosing $ \delta = \frac{q}{4KC},$ we have

   $$ -\frac{q}{2}  \leq |Z(T_{x_{i}}g)(t,\xi)|^{2}-|Z(T_{x_{i}}g)(t',\xi ')|^{2}  \leq \frac{q}{2} $$
   $$ |Z(T_{x_{i}}g)(t',\xi ')|^{2} -\frac{q}{2}  \leq |Z(T_{x_{i}}g)(t,\xi)|^{2}  \leq |Z(T_{x_{i}}g)(t',\xi ')|^{2} +\frac{q}{2}. $$
   Thus, 
   \begin{align}{\label{mesheq}} \begin{split}\min_{(t', \xi') \in M_{\delta}} \sum_{i=1}^{m} |Z(T_{x_{i}}g)(t',\xi ')|^{2}-\frac{mq}{2} & \leq \sum_{i=1}^{m} |Z(T_{x_{i}}g)(t,\xi)|^{2} \\ &\leq \max_{(t', \xi') \in M_{\delta}} \sum_{i=1}^{m} |Z(T_{x_{i}}g)(t',\xi ')|^{2}+\frac{mq}{2}\end{split}\end{align}
   for all $(t,\xi) \in Q.$ \par
   
   \textbf{Probability Bounds}: We estimate the probability of the events E$_{1}^{c}$ and E$_{2}^{c}$ and show that they can be made sufficiently small by increasing the number of samples, which, in turn, will show that the event E holds with high probability.
   \begin{align*}
       \mathbb{P}(\mbox{E}_{1}^{c}) &\leq\mathbb{P}\left(\esssup_{(t,\xi)\in Q}\sum_{i=1}^{m} |Z(T_{x_i}g)(t,\xi)|^{2} > B\right) \\
       & \leq \mathbb{P}\left(\max_{(t', \xi') \in M_{\delta}} \sum_{i=1}^{m} |Z(T_{x_{i}}g)(t',\xi ')|^{2}+\frac{mq}{2} > B\right) \hspace{2cm} \mbox{using (\ref{mesheq})}\\
       & \leq \sum_{(t',\xi') \in M_{\delta}}\mathbb{P}\left(\sum_{i=1}^{m} |Z(T_{x_{i}}g)(t',\xi ')|^{2}+\frac{mq}{2} > B\right)\\
       &=\sum_{(t',\xi') \in M_{\delta}}\mathbb{P}\left( \frac{1}{m}\sum_{i=1}^{m} |Z(T_{x_{i}}g)(t',\xi ')|^{2} -\sum_{l \in \ZZ}|\widehat{g}(\xi' +l)|^{2} > \beta-\frac{q}{2}-\sum_{l \in \ZZ}|\widehat{g}(\xi' +l)|^{2}\right)\\
       & \leq \sum_{(t',\xi') \in M_{\delta}}\mathbb{P}\left(\frac{1}{m}\sum_{i=1}^{m} |Z(T_{x_{i}}g)(t',\xi ')|^{2} -\sum_{l \in \ZZ}|\widehat{g}(\xi' +l)|^{2} > \beta-\frac{q}{2}\right)\\
   \end{align*}
   On application of Lemma \ref{hoeffding}, we get
$$\mathbb{P}(\mbox{E}_{1}^{c})  \leq \left(\frac{4CK}{q}+1\right)^{2} \exp{\left(\frac{-2m(\beta-\frac{q}{2})^2}{K^4}\right)}. $$
   Now, $$\left(\frac{4CK}{q}+1\right)^{2} \exp{\left(\frac{-2m(\beta-\frac{q}{2})^2}{K^4}\right)} < \frac{\varepsilon}{2},  \mbox{ if }  m > \frac{K^4}{2(\beta-\frac{q}{2})^{2}} \ln{\left(\frac{2(\frac{4CK}{q}+1)^{2}}{\varepsilon}\right)}.$$
   Therefore, 
   $$ \mathbb{P}(\mbox{E}_{1}^{c}) < \frac{\varepsilon}{2}, \mbox{ when }  m > \frac{K^4}{2(\beta-\frac{q}{2})^{2}} \ln{\left(\frac{2(\frac{4CK}{q}+1)^{2}}{\varepsilon}\right)}.$$
   
  \par Similarly,
   \begin{align*}
       \mathbb{P}(\mbox{E}_{2}^{c}) &\leq\mathbb{P}\left(\essinf_{(t,\xi)\in Q}\sum_{i=1}^{m} |Z(T_{x_i}g)(t,\xi)|^{2} < A\right) \\
       & \leq \mathbb{P}\left(\min_{(t', \xi') \in M_{\delta}} \sum_{i=1}^{m} |Z(T_{x_{i}}g)(t',\xi ')|^{2}-\frac{mq}{2} < A\right) \hspace{2cm} \mbox{using (\ref{mesheq})}\\
       & \leq \sum_{(t',\xi') \in M_{\delta}}\mathbb{P}\left( \sum_{i=1}^{m} |Z(T_{x_{i}}g)(t',\xi ')|^{2}-\frac{mq}{2} < A\right)\\
       &=\sum_{(t',\xi') \in M_{\delta}}\mathbb{P}\left( \frac{1}{m}\sum_{i=1}^{m} |Z(T_{x_{i}}g)(t',\xi ')|^{2} -\sum_{l \in \ZZ}|\widehat{g}(\xi' +l)|^{2}< \alpha+\frac{q}{2}-\sum_{l \in \ZZ}|\widehat{g}(\xi' +l)|^{2}\right)\\
       & = \sum_{(t',\xi') \in M_{\delta}}\mathbb{P}\left(\frac{1}{m}\sum_{i=1}^{m} -|Z(T_{x_{i}}g)(t',\xi ')|^{2} +\sum_{l \in \ZZ}|\widehat{g}(\xi' +l)|^{2}> -\alpha-\frac{q}{2}+\sum_{l \in \ZZ}|\widehat{g}(\xi' +l)|^{2}\right)\\
       \end{align*}
       Now using Assumption 2, followed by an application of Lemma \ref{hoeffding}, we get 
       \begin{align*}
      \mathbb{P}(\mbox{E}_{2}^{c}) & \leq \sum_{(t',\xi') \in M_{\delta}}\mathbb{P}\left( \frac{1}{m}\sum_{i=1}^{m} -|Z(T_{x_{i}}g)(t',\xi ')|^{2} +\sum_{l \in \ZZ}|\widehat{g}(\xi' +l)|^{2}> -\alpha+\frac{q}{2}\right)\\ 
       & \leq \left(\frac{4CK}{q}+1\right)^{2} \exp{\left(\frac{-2m(\alpha-\frac{q}{2})^2}{K^4}\right)}.
   \end{align*}
   $$ \left(\frac{4CK}{q}+1\right)^{2} \exp{\left(\frac{-2m(\alpha-\frac{q}{2})^2}{K^4}\right)}  < \frac{\varepsilon}{2} \mbox{, if } m > \frac{K^4}{2(\alpha-\frac{q}{2})^{2}} \ln{\left(\frac{2(\frac{4CK}{q}+1)^{2}}{\varepsilon}\right)}$$
   Thus, 
   $$\mathbb{P}(\mbox{E}_{2}^{c}) < \frac{\varepsilon}{2}, \mbox{ when }  m > \frac{K^4}{2(\alpha-\frac{q}{2})^{2}} \ln{\left(\frac{2(\frac{4CK}{q}+1)^{2}}{\varepsilon}\right)}$$
   Having the probability estimates at our disposal for the events E$_{1}^{c}$ and E$_{2}^{c},$ we evaluate the probability for the event E.

   If $m > \max\{\frac{K^4}{2(\alpha-\frac{q}{2})^{2}} \ln{\left(\frac{2}{\varepsilon}(\frac{4CK}{q}+1)^{2}\right)}, \frac{K^4}{2(\beta-\frac{q}{2})^{2}} \ln{\left(\frac{2}{\varepsilon}(\frac{4CK}{q}+1)^{2}\right)}\},$ then
   \begin{align*}
       \mathbb{P}(\mbox{E})=\mathbb{P}(\mbox{E}_{1} \cap \mbox{E}_{2}) &=1-\mathbb{P}(\mbox{E}_{1}^{c} \cup \mbox{E}_{2}^{c})\\
       & \geq 1- \mathbb{P}(\mbox{E}_{1}^{c})-\mathbb{P}(\mbox{E}_{2}^{c})\\
       &> 1-\varepsilon.
       \end{align*}
       Hence, the frame inequality 
       \begin{equation*}
    m \alpha \|f\|^{2} \leq \sum_{i=1}^{m} \sum_{k,n \in \ZZ} |\langle f,M_{k}T_{n+{x_{i}}}g \rangle|^{2} \leq m \beta \|f\|^{2}
\end{equation*}
holds for all $f \in L^2(\RR)$ with probability atleast $1-\varepsilon.$

\end{proof}
 \section{Examples}
 A function $f$ defined on $\RR$ is said to be of \textbf{moderate decrease} if $f$ is continuous and there exists $A, \varepsilon >0$ such that 
 \begin{equation} \label{moddec}|f(t)| \leq \frac{A}{1+|t|^{1+\varepsilon}}~~~ \mbox{ for all } t \in \RR.\end{equation}
 Also, let $\Phi_{f}$ denote the periodic square summability of $f$ in the frequency domain, i.e., 
 $$\Phi_{f}(t)=\sum_{l \in \ZZ}|f(t+l)|^2, ~~ t \in [0,1].$$
 If the functions $f$ and $\widehat{f}$ are of moderate decrease, then due to their sufficient decay, Assumption 1 is satisfied. Moreover, the right-hand side inequality in Assumption 2 also follows. Now note that since $\widehat{f}$ is of moderate decrease, the function $\Phi_{f}$ is continuous. Hence, to ensure left-hand side of Assumption 2, it is enough to have that $\Phi_{f}$ is non-zero everywhere on the unit interval $[0,1].$  
 
 A straightforward application of Taylor's theorem tells us that if the partial derivatives of the real and imaginary part of the function $Zf(t,\xi)$ are bounded in $[0,1] \times [0,1],$ then the Assumption 3 is satisfied. Furthermore, if the function $f$ satisfies an additional decay
 \begin{equation} \label{suffdec}|f(\xi)| \leq \frac{B}{1+|\xi|^{2 + \varepsilon}} ~ \mbox{ and } ~ |\widehat{f}(\xi)| \leq \frac{C}{1+|\xi|^{2 + \delta}} ~ \mbox{ for all } \xi \in \RR,\end{equation} for some $B,C, \varepsilon, \delta>0$, then the existence and boundedness of partial derivatives of $Z{f}$ follows as a consequence.

 The following is a list of examples that satisfy the assumptions of the main theorem:
\begin{itemize}
    \item \textbf{Hermite functions:} These functions are defined as $h_{n}(x)=e^{\pi x^{2}} \frac{\mathrm{d}^{n}}{\mathrm{d}x^{n}}(e^{-2 \pi x^2}).$ Notably, they exhibit polynomial decay of every order due to the presence of a Gaussian factor. Also, note that these functions are invariant, upto a constant factor, under the operation of the Fourier transform. Consequently, they have finitely many zeroes on the real line. Moreover, conclusively, they satisfy (\ref{suffdec}). Hence, these functions satisfy all the assumptions and serve as a model example for our theorem. A careful reader might observe that no specific properties of Hermite functions have been explicitly used to verify that they satisfy the necessary assumptions. The key requirements are simply their sufficient decay and the fact that their Fourier transform has finitely many zeros. These insights naturally extend to a broader class of examples. 
    \item \textbf{Schwartz class functions with finitely many zeroes:} The inverse Fourier transform of a Schwartz class function that has only finitely many zeroes will satisfy all the assumptions.
     \item \textbf{Totally positive functions:} We use an equivalent definition of totally positive functions, which is a characterization due to Schoenberg \cite{Schoenberg1}. Also see \cite{Grochenigtpoverrationallattice}. An integrable function $f \in L^{2}(\RR)$ is said to be of totally positive if $f$ possesses the factorization
$$\widehat{f}(\xi)=c e^{-\gamma \xi^2}e^{2 \pi i \nu \xi} \prod_{j=1}^{N} (1+ 2 \pi i \nu_{j}\xi)^{-1}e^{-2 \pi i \nu_{j} \xi},$$
where $\gamma \geq 0,c>0,\nu, \nu_{j} \in \RR, N \in \NN \cup \{\infty\}.$ If $ \gamma > 0,$ then the decay conditions in the assumptions are met for all $N \in \NN,$ and when $\gamma =0,$ the assumptions are met for $N > 2.$ Indeed, the decay in the frequency domain is apparent, whereas in the time domain, it follows from a theorem of Paley and Wiener. One can also note that these functions $(\gamma =0, N \in \NN)$ lie outside of Schwartz class.
    \item \textbf{B-splines:}  The $ n $-th order B-splines are defined as the $(n+1)$-fold convolution of the characteristic function on the unit interval $[0,1]$. We denote them by $ \beta^n $, explicitly given by:  
\[
\beta^n(t) = \underbrace{\left( \chi_{[0,1]} \ast \ldots \ast \chi_{[0,1]} \right)}_{(n+1) \text{-fold convolution}}(t).
\]
An important property of $ \beta^n(t) $ is that its Fourier transform $\widehat{\beta^n}(\xi)$ exhibits a decay rate bounded by $ |\xi|^{-(n+1)} $. This ensures that $\widehat{\beta^n}$  possesses sufficient decay, allowing it to meet the assumptions required in the main theorem.

\end{itemize}
\section*{Acknowledgement}
SS acknowledges the Anusandhan National Research Foundation, Government of India, for the financial support through MATRICS grant - MTR/2022/000383.

\bibliographystyle{abbrv}\bibliography{main.bib}

\end{document}